\documentclass[12pt,a4paper]{article}

\usepackage[T1]{fontenc}
\usepackage[utf8]{inputenc}
\usepackage[british]{babel}
\usepackage{amsmath,amssymb,amsthm}
\usepackage{url}

\newcommand\N{\mathbb{N}}
\newcommand\Z{\mathbb{Z}}

\newtheorem{theorem}{Theorem}[section]
\newtheorem{corollary}[theorem]{Corollary}
\newtheorem{lemma}[theorem]{Lemma}
\newtheorem{proposition}[theorem]{Proposition}
\newtheorem{theorema}{Theorem}

\theoremstyle{remark}
\newtheorem{remark}[theorem]{Remark}

\DeclareMathOperator\Aut{Aut}
\DeclareMathOperator\Ker{Ker}
\DeclareMathOperator\id{id}
\DeclareMathOperator\Soc{Soc}
\title{From actions of an abelian group on itself to left braces}
\author{A. Ballester-Bolinches\thanks{Departament de Matem{\`a}tiques, Universitat de Val{\`e}ncia; Dr.\ Moliner, 50; 46100 Burjassot, Val\`encia, Spain; \texttt{Adolfo.Ballester@uv.es}; ORCID 0000-0002-2051-9075} \and R. Esteban-Romero\thanks{Departament de Matem{\`a}tiques, Universitat de Val{\`e}ncia; Dr.\ Moliner, 50; 46100 Burjassot, Val\`encia, Spain; \texttt{Ramon.Esteban@uv.es}; ORCID 0000-0002-2321-8139}\and L. A. Kurdachenko\thanks{Department of Algebra and Geometry, Oles Honchar Dnipro National University, Dnipro 49010, Ukraine; \texttt{lkurdachenko@gmail.com}; ORCID 0000-0002-6368-7319}\ \thanks{Departament de Matem{\`a}tiques, Universitat de Val{\`e}ncia; Dr.\ Moliner, 50; 46100 Burjassot, Val\`encia, Spain} \and V. P{\'e}rez-Calabuig\thanks{Departament de Matem{\`a}tiques, Universitat de Val{\`e}ncia; Dr.\ Moliner, 50; 46100 Burjassot, Val\`encia, Spain; \texttt{Vicent.Perez-Calabuig@uv.es}; ORCID 0000-0003-4101-8656}}
\date{}
\begin{document}
\maketitle

\begin{abstract}
  We present a construction of left braces of right nilpotency class at most two based on suitable actions of an abelian group on itself with an invariance condition. This construction allows us to recover the construction of a free right nilpotent one-generated left brace of class two.

  \emph{Mathematics Subject Classification (2020):} 16T25, % Yang-Baxter equations
  16N40, % Nil and nilpotent radicals, sets, ideals, associative rings
  81R50 % 	Quantum groups and related algebraic methods applied to problems in quantum theory [See also 16T20, 17B37]

  \emph{Keywords:} left brace, right nilpotent left brace, one-generated left brace, Smoktunowicz-nilpotent left brace.
\end{abstract}

\section{Introduction}
A \emph{skew left brace} $(A, {+}, {\cdot})$ consists of a set $A$ with two binary operations such that $(A, {+})$ and $(A, {\cdot})$ are groups related by the modified distributive law $a(b+c)=ab-a+ac$ for all $a$, $b$, $c\in A$. If, in addition, the group $(A, {+})$ is abelian, we speak of a \emph{skew left brace of abelian type} or, simply, a \emph{left brace}. When the operations $+$ and $\cdot$ are clear from the context, we will refer to the left brace $(A, {+}, {\cdot})$ simply as $A$. The structure of left brace was introduced by Rump as a generalisation of Jacobson radical rings in \cite{Rump07}. Guarnieri and Vendramin \cite{GuarnieriVendramin17} introduced skew left braces by removing the commutativity of the addition in Rump's left braces. An interesting survey on left braces can be found in the paper of Ced{\'o} \cite{Cedo18}, where most of the definitions and results presented in this section can be found.

The structure of a left brace depends critically on the lambda action of the multiplicative group on the additive group. Let $a$, $b$ be elements of a left brace $A$. We define $\lambda_a(b)=-a+ab$. It is well known that $\lambda\colon (A, {\cdot})\longrightarrow \Aut(A, {+})$ given by $\lambda(a)=\lambda_a$ for $a\in A$ is a group homomorphism. 

Another interesting operation that gives structural information on a left brace is the star operation. We define it as $a*b=-a+ab-b=\lambda_a(b)-b$ for $a$, $b\in B$. Given two subsets $L$ and $M$ of $A$, we define $L*M$ as the subgroup of the additive group generated by $\{l*m\mid l\in L, m\in M\}$. We set $A^{(1)}=A$, $A^{(r+1)}=A^{(r)}*A$ for a natural $r$, and $A^1=A$, $A^{r+1}=A*A^r$ for a natural $r$.

These constructions allow us to define two different nilpotencies for left braces: we say that a left brace $A$ is \emph{left nilpotent} if there exists a natural $m$ such that $A^{m+1}=\{0\}$, and that $A$ is \emph{right nilpotent} if there exists a natural $m$ such that $A^{(m+1)}=\{0\}$. If $A$ is right nilpotent (respectively, left nilpotent), the smallest $m$ satisfying $A^{(m+1)}=\{0\}$, but $A^{(m)}\ne \{0\}$ (respectively, $A^{m+1}=\{0\}$, but $A^m\ne \{0\}$) is called the \emph{right} (respectively, \emph{left}) \emph{nilpotency class} of $A$. A left brace that is both left nilpotent and right nilpotent is said to be \emph{nilpotent in the sense of Smoktunowicz} or, simply, \emph{Smoktunowicz-nilpotent}. These left braces were studied by Smoktunowicz in \cite{Smoktunowicz18-tams}. For left braces, this concept coincides with the central nilpotency introduced in \cite{BonattoJedlicka-central-nilp-skew-braces} for skew left braces (see also~\cite{JespersVanAntwerpenVendramin22}).

We say that a subset $S$ of a left brace $(A, {+}, {\cdot})$ is a \emph{subbrace} of $(A, {+}, {\cdot})$ if $S$ is a subgroup of $(A, {+})$ and of $(A, {\cdot})$. The intersection of a family of subbraces of a given left brace $(A, {+}, {\cdot})$ is again a subbrace of $(A, {+}, {\cdot})$. If $X$ is a subset of $(A, {+}, {\cdot})$, the intersection of all subbraces of $(A, {+}, {\cdot})$ is a subbrace of $(A, {+}, {\cdot})$ called the \emph{subbrace generated by} $X$. When $X=\{a\}$ is a singleton, we simply speak of the \emph{subbrace generated by~$a$}. The description of the subbrace generated by a set or the free object on a set $X$ in the category of left braces seems to be a difficult problem in general. In fact, Agata and Alicja Smoktunowicz ask about such description in the class of left braces of right nilpotency class at most two \cite[Question~5.5]{SmoktunowiczSmoktunowicz18}. % This is the reason we start to solve this problem in easy cases, like when $A^{(3)}=\{0\}$.

Stefanello and Trappeniers \cite[Theorem~3.13]{StefanelloTrappeniers23-jpaa} have given some interesting characterisations of left braces of right nilpotency class at most two. We present here their results in the form we need in the sequel.

\begin{proposition}\label{prop-3.13}
  Let $A$ be a left brace. The following statements are equivalent in pairs:
  \begin{enumerate}
  \item $A^{(3)}=\{0\}$.\label{en-ST-1}
  \item If $z\in A^{(2)}$ and $a\in A$, then $z*a=0$.\label{en-ST-1b}
  \item $A^{(2)}\subseteq \Ker\lambda$.\label{en-ST-2}
  \item $\lambda$ is a group homomorphism between $(A, {+})$ and $\Aut(A, {+})$.\label{en-ST-3}
  \item $\lambda_{\lambda_a(b)}=\lambda_b$ for every $a$, $b\in A$.\label{en-ST-4}
  \end{enumerate}
\end{proposition}

The left braces satisfying Condition~\ref{en-ST-3} of Proposition~\ref{prop-3.13} have been studied by Bardakov, Neshchadim, and Yadav in \cite{BardakovNeshchadimYadav22-jpaa}, who called them \emph{$\lambda$-homomorphic}. They turn out to be examples of \emph{bi-skew left braces} or \emph{symmetric skew left braces}, that is, skew left braces $(A, {+}, {\cdot})$ for which $(A, {\cdot}, {+})$ is again a skew left brace, studied for the first time by Childs in \cite{Childs19} and later by Caranti \cite{Caranti20}, Koch \cite{Koch21}, and Caranti and Stefanello \cite{CarantiStefanello21}. We will prove Proposition~\ref{prop-3.13} in Section~\ref{sec-prelim} for the sake of completeness.

A subbrace $S$ of a left brace $A$ is called a \emph{left ideal} of $A$ if $A*S\subseteq S$. This is equivalent to saying that $\lambda_a(S)\subseteq S$ for each $a\in S$. A subbrace $S$ of a left brace $A$ is called an \emph{ideal} of $A$ if $A*S\subseteq S$ and $S*A\subseteq A$. This is equivalent to saying that $S$ is a left ideal of $A$ and $S$ is a normal subgroup of the multiplicative group of~$S$. We obtain that $A^{(r)}$ is an ideal of $A$ for all naturals $r$ and that $A^r$ is a left ideal of $A$ for all naturals~$r$. The \emph{socle} of $A$, defined as the kernel of $\lambda$ in the case of left braces, is another example of an ideal of~$A$.

The first result of this paper is a general construction of a left brace from an abelian group $(A, {+})$ and an action of $(A, {+})$ on itself that satisfies an invariance condition from which the universal constructions related to one-generated subbraces of right nilpotent left braces of class at most two directly follow. 
\begin{theorema}\label{th-mod-brace}
  Suppose that $(B, {+})$ is an abelian group acting on itself by means of an action $\lambda\colon (B, {+})\longrightarrow \Aut(B, {+})$ given by $\lambda(x)=\lambda_x$ for $x\in B$. Suppose also that if $x$, $y\in B$, then $\lambda_{\lambda_y(x)}=\lambda_x$. We can define a product $\cdot$ on $B$ by means of $xy=x+\lambda_x(y)$ for $x$, $y\in B$. Then $(B, {+}, {\cdot})$ is a left brace with $B^{(3)}=\{0\}$.
\end{theorema}

% We give the proof of this result in Section~\ref{sec-mod-braces}. It is clear that, a
As a consequence of Proposition~\ref{prop-3.13}, all left braces of right nilpotency class two can be described in this way.

% Some universal constructions of one-generated subbraces of right nilpotent braces of right nilpotency class at most two can be obtained as a consequence of Theorem~\ref{th-mod-brace}.

Rump \cite{Rump20}, answering the above-mentioned question of Smoktunowicz and Smoktunowicz, has given a description of the free one-generated right nilpotent left brace of class two and a description of subbrace generated by an element of a right nilpotent left brace of class at most two.

\begin{theorema}[Rump {\cite{Rump20}}]\label{th-subbrace-1gen}
Let $(A, {+}, {\cdot})$ be a left brace with $A^{(3)}=\{0\}$. Let $a\in A$. For $i\in\Z$, we define $a_i=\lambda_{a^i}(a)=-a^i+a^{i+1}$. In particular, $a_0=-a^0+a^1=-0+a=a$. We observe that $\lambda_{a^j}(a_i)=\lambda_{a^j}(\lambda_{a^i}(a))=\lambda_{a^ja^i}(a)=\lambda_{a^{i+j}}(a)=a_{i+j}$.
  The set
  \[B=\Bigl\{\sum_{i\in\Z}x_ia_i\mid x_i\in\Z, i\in\Z, \text{ and $\{i\in\Z\mid x_i\ne 0\}$ is finite}\Bigr\}.\]
coincides with the subbrace of $A$ generated by~$a\in A$.
  Furthermore, if $D=\{\sum_{i\in\Z}x_ia_i\in B\mid \sum_{i\in\Z}x_i=0\}$, then $B*B=D$.
\end{theorema}

% We give an alternative proof of Theorem~\ref{th-subbrace-1gen}. 
% As an application of Theorem~\ref{th-mod-brace}, we can recover the construction of a free one-generated left brace $C$ with $C^{(3)}=\{0\}$ given by Rump \cite{Rump20}. 

\begin{theorema}[see Rump {\cite{Rump20}}]\label{th-free-A(3)}
  Let $C$ be a free abelian group with basis $\{c_i\mid i\in\Z\}$. Given $x\in C$, then $x=\sum_{i\in \Z}x_ic_i$ with $x_i\in\Z$ for all $i\in\Z$ and $x_i\ne 0$ for a finite number of $i\in\Z$. We define an action of $x\in C$ on $C$ by means of $\lambda_x(y)=\sum_{i\in\Z}y_ic_{i+\sum_{j\in \Z}x_j}=\sum_{i\in\Z}y_{i-\sum_{j\in\Z}x_j}c_i$ for $y=\sum_{i\in\Z}y_ic_i$. Then $(C, {+}, {\cdot})$ with the multiplication given by $xy=x+\lambda_x(y)$ for $x$, $y\in C$ is a left brace generated by $a=c_0$ with $C^{(3)}=\{0\}$ and $C^j\ne \{0\}$ for all $j\in\N$. Moreover, if $A$ is a left brace with $A^{(3)}=\{0\}$ and $b\in A$, then there exists a unique left brace epimorphism $\alpha$ from $C$ and the subbrace of $A$ generated by $b$ such that $\alpha(c_0)=b$.
\end{theorema}

Rump's constructions rely on other algebraic structures, like cycle sets or q-braces. We recover Rump's results with a different approach that strongly depend on Proposition~\ref{prop-3.13}  and Theorem~\ref{th-mod-brace} (Sections~\ref{sec-A(3)} and~\ref{sec-A(3)-free}).

%respectively, based on the general construction of right nilpotent left braces of right nilpotency class at most~$2$ of Theorem~\ref{th-mod-brace} from a suitable actions of an abelian group on itself. The fact that this construction generates a right nilpotent left brace of class at most~$2$ depends on Proposition~\ref{prop-3.13} and other preliminary results that we show in Section~\ref{sec-prelim}.

The special case of Smoktunowicz-nilpotent left braces $A$ with $A^{(3)}=\{0\}$ and $A^{m+1}=\{0\}$ for a natural number $m\ge 2$ becomes particularly relevant. Of course, we can apply the descriptions of the subbrace $S$ generated by one element  corresponding to $A^{(3)}=\{0\}$ and the free one-generated left brace (Theorems~\ref{th-subbrace-1gen} and~\ref{th-free-A(3)}) to this case, but in this case it seems desirable to have a description that takes into account the left nilpotency class and the terms $S^r$ for all natural numbers $r$ become apparent from the description. An adaptation of the previous description to the subbrace generated by an element in these left braces appears in Section~\ref{sec-Smoktunowicz}. We prove:
\begin{theorema}\label{th-brace-1gen}
  Let $A$ be a left brace such that $A^{(3)}=A^{m+1}=\{0\}$ for a natural number~$m\ge 2$ and let $a\in A$. Define $a_1=a$, $a_{j+1}=a*a_j$ for $j\ge 1$. Then the set
  \[S=\Bigl\{\sum_{k=1}^mt_ka_k\mid t_k\in \mathbb{Z}, 1\le k\le m\Bigr\}\]
  is the subbrace of $A$ generated by~$a$.
\end{theorema}

The construction of the free one-generated Smoktunowicz-nilpotent left braces with right nilpotency class two and a given left nilpotency class is given in Section~\ref{sec-Smoktunowicz-free}. We prove:
\begin{theorema}\label{th-Smoktunowicz-free}
  Let $m\ge 2$ be a natural number. Consider in $B_m=\overbrace{\Z\times\dots\times \Z}^{(m)}$ the usual addition given by
  \[(n_1,\dots, n_m)+(t_1,\dots, t_m)=(n_1+t_1,\dots, n_m+t_m)\]
  and the multiplication
  \[(n_1,\dots, n_m)(t_1,\dots, t_m)=(w_1,\dots, w_m)\]
  where, for $1\le i\le m$,
  \[w_i=n_i+\sum_{k=0}^{i-1}\binom{n_1}{k}t_{i-k}.\]
  Then $(B, {+}, {\cdot})$ is a left brace with $B_m^{(3)}=B_m^{m+1}=\{0\}$, $B_m^{m}\ne \{0\}$; $B$ is generated by $\mathbf{b}=(1,0,\dots, 0)$, and given a left brace $(A, {+}, {\cdot})$ with $A^{(3)}=A^{m+1}=0$ and $a\in A$, there exists a unique left brace homomorphism from $B_m$ to $A$ mapping $\mathbf{b}$ to $a$ and whose image is the subbrace of~$A$ generated by $a$.
\end{theorema}
The construction of this free left brace also appears as a particular case of Theorem~\ref{th-mod-brace}.

\section{Preliminary results}\label{sec-prelim}

For the sake of completeness, we present here a proof of Proposition~\ref{prop-3.13}.

\begin{proof}[Proof of Proposition~\ref{prop-3.13}]
  \emph{\ref{en-ST-1} implies~\ref{en-ST-1b}}. Suppose that $z\in A^{(2)}$ and $a\in A$. Then $z*a=A^{(3)}=\{0\}$.

  \emph{\ref{en-ST-1b} implies~\ref{en-ST-2}}. Let $z\in A^{(2)}$ and $a\in A$. Since $\lambda_z(a)-a=z*a=0$, we have that $\lambda_z(a)=a$. This implies that $\lambda_z=\id$ and so $z\in \Ker\lambda$.

% \emph{\ref{en-ST-1} implies~\ref{en-ST-2}}. Suppose that $A^{(3)}=\{0\}$. Given $a$, $b$, $c\in A$, we have that $0=(a*b)*c=\lambda_{a*b}(c)-c$, what implies that $\lambda_{a*b}(c)=c$. Consequently, $a*b\in \Ker\lambda$. Since $\Ker \lambda=\Soc A$, the socle of $A$, is an ideal of $A$ by a result of Rump (see \cite[Proposition~2.15]{Cedo18}), we have that $A^{(2)}$, that is an ideal of $A$, is also contained in $\Ker\lambda$.

  \emph{\ref{en-ST-2} implies~\ref{en-ST-3}}. Since $ab=a*b+a+b=(a*b)\lambda_{a*b}^{-1}(a+b)$, we have that $\lambda_{a}\circ\lambda_b=\lambda_{ab}=\lambda_{(a*b)\lambda_{a*b}^{-1}(a+b)}=\lambda_{a*b}\circ\lambda_{\lambda_{a*b}^{-1}(a+b)}=\lambda_{a+b}$ and so $\lambda\colon (A, {+})\longrightarrow \Aut(A, {+})$ is a group homomorphism.

  \emph{\ref{en-ST-3} implies~\ref{en-ST-4}}. We have that $\lambda_a\circ\lambda_b=\lambda_{ab}=\lambda_{a+\lambda_a(b)}=\lambda_a\circ\lambda_{\lambda_a(b)}$. Therefore, $\lambda_b=\lambda_a(b)$.

  \emph{\ref{en-ST-4} implies~\ref{en-ST-2}}. Let $a$, $b\in A$. We have that
  \begin{align*}
    \lambda_{a*b}&=\lambda_{\lambda_a(b)-b}=\lambda_{\lambda_a(b)\lambda^{-1}_{\lambda_a(b)}(-b)}=\lambda_{\lambda_a(b)\lambda_b^{-1}(-b)}=\lambda_{\lambda_a(b)}\circ\lambda_{\lambda_b^{-1}(-b)}\\
                            &=\lambda_b\circ\lambda_{\lambda_b^{-1}}=\lambda_{b\lambda_b^{-1}(-b)}=\lambda_{b-b}=\lambda_0=\id.                              
  \end{align*}
  Therefore, $a*b\in \Ker\lambda$. Since $\Ker\lambda=\Soc A$ is an ideal of $A$, we have that $A^{(2)}\subseteq \Ker\lambda$.

  \emph{\ref{en-ST-2} implies~\ref{en-ST-1}}. Suppose that $A^{(2)}\subseteq \Ker\lambda$. It is enough to show that all generators of $A^{(3)}$, that have the form $d*e$ for $d\in A^{(2)}$ and $e\in A$, are zero. Since $A^{(2)}\subseteq \Ker \lambda$, we have that $d*e=\lambda_d(e)-e=e-e=0$ if $d\in A^{(2)}$ and $e\in A$. Therefore, $A^{(3)}=\{0\}$.
\end{proof}

The lambda map of a skew left brace is always an action from the multiplicative group of the left brace on the additive group. Our construction of left braces in Theorem~\ref{th-mod-brace} is based on an action from the additive group on itself. The following result shows the relation between both actions.
\begin{proposition}
  Let $(A, {+})$ be an abelian group, $\lambda\colon A\longrightarrow \Aut(A, {+})$ and define a binary operation $\cdot$ on $A$ by means of $ab=a+\lambda_a(b)$ (this happens when $\lambda$ is the map defined on a left brace $(A, {+}, {\cdot})$ by $\lambda_a(b)=-a+ab$ for $a$, $b\in A$). If two of the following three statements hold, then so does the other one:
  \begin{enumerate}
  \item $\lambda_{a+b}=\lambda_a\circ\lambda_b$ for every $a$, $b\in A$ (that is, $\lambda\colon (A, {+})\longrightarrow \Aut(A, {+})$ is a group homomorphism).\label{en-21-1}
  \item $\lambda_{ab}=\lambda_a\circ\lambda_b$ for every $a$, $b\in A$.\label{en-21-2}
  \item $\lambda_{\lambda_a(b)}=\lambda_b$ for every $a$, $b\in A$.\label{en-21-3}
  \end{enumerate}
\end{proposition}
\begin{proof}
  \emph{\ref{en-21-1} and~\ref{en-21-2} imply~\ref{en-21-3}}. Since $\lambda_a\circ\lambda_b=\lambda_{ab}=\lambda_{a+\lambda_a(b)}=\lambda_a\circ\lambda_{\lambda_a(b)}$, we conclude that $\lambda_b=\lambda_{\lambda_a(b)}$.

  \emph{\ref{en-21-2} and~\ref{en-21-3} imply~\ref{en-21-1}}. We have that
  $\lambda_{a+b}=\lambda_{a\lambda_a^{-1}(b)}=\lambda_a\circ\lambda_{\lambda_a^{-1}(b)}=\lambda_a\circ\lambda_{\lambda_a(\lambda_a^{-1}(b))}=\lambda_a\circ\lambda_b$.

  \emph{\ref{en-21-3} and~\ref{en-21-1} imply~\ref{en-21-2}}. We have that $\lambda_{ab}=\lambda_{a+\lambda_a(b)}=\lambda_a\circ\lambda_{\lambda_a(b)}=\lambda_a\circ\lambda_b$.
\end{proof}

Some of the results of Section~\ref{sec-Smoktunowicz} will be given in terms of generalised binomial coefficients of the form $\binom{n}{k}$, where $n$ is an integer, that could be eventually negative, and $k$ is a non-negative integer. They are defined by means of
\[\binom{n}{0}=1,\qquad\binom{n}{k}=\frac{n(n-1)\dotsm(n-k+1)}{k!}\quad \text{if $k>0$}.\]
Some of the well-known identities of the usual binomial coefficients are also satisfied by the generalised binomial coefficients. The next lemma summarises some of these identities (see, for example, \cite{Gould72}).
\begin{lemma}
  \begin{enumerate}
  \item   Let $n\in\mathbb{Z}$ and $k$ a non-negative integer. Then:
    \[\binom{n}{k}+\binom{n}{k+1}=\binom{n+1}{k+1}.\]
  \item (Chu-Vandermonde's Theorem) Let $n$ and $t$ be integer numbers and $r$ a non-negative integer. Then 
    \[\binom{n+t}{r}=\sum_{k=0}^r\binom{n}{k}\binom{t}{r-k}.\]
  \end{enumerate}
\end{lemma}

The following properties of the star operation can be found, for instance, in \cite[Lemma~2.1]{JespersVanAntwerpenVendramin22}. In the absence of parentheses, the star operation will be performed before the additions.
\begin{lemma}\label{lemma-4.2}
    Let $(A, {+}, {\cdot})$ be a left brace. Then, the following identities hold for $a$, $b$, $c\in A$:
  \begin{enumerate}
  \item $a*(b+c)=a*b+a*c$.
  \item $(ab)*c=a*(b*c)+b*c+a*c$.\label{en-lemma-4.2b}
  \end{enumerate}
\end{lemma}

For left braces $A$ with $A^{(3)}=\{0\}$, we have the following result.

\begin{proposition}\label{prop-1.5}
  Let $A$ be a left brace such that $A^{(3)}=\{0\}$. Then, for every $x$, $y$, $z\in A$, we have that
  \[x*(z*y)+x*y+z*y=(x+z)*y=(xz)*y.\]
\end{proposition}
\begin{proof}
  The first equality is just Lemma~\ref{lemma-4.2}~(\ref{en-lemma-4.2b}). The second equality follows from the fact that $\lambda_{x+y}=\lambda_x\lambda_y$ by the equivalence between~\ref{en-ST-1} and~\ref{en-ST-3} in Proposition~\ref{prop-3.13} and so $(x+z)*y=\lambda_{x+z}(y)-y=\lambda_x(\lambda_z(y))-y=\lambda_{xz}(y)-y=(xz)*y$.
\end{proof}

\begin{corollary}\label{corol-1.6}
  Let $A$ be a left brace with $A^{(3)}=\{0\}$. Then, for every $x$, $y\in A$, $b\in A^{(2)}$, we have that $(x+b)*y=x*y$.
\end{corollary}
\begin{proof}
  We have that $b*y=0$ by Proposition~\ref{prop-3.13}.
  By Proposition~\ref{prop-1.5}, we have that
  \[(x+b)*y=x*(b*y)+x*y+b*y=x*0+x*y+0=x*y.\qedhere\]
\end{proof}

Now we present some general results about left braces $A$ satisfying the condition $A^{(3)}=\{0\}$. They will be used in the description of the subbrace generated by an element of~$A$.
\begin{proposition}\label{prop-1.7}
    Let $A$ be a left brace and suppose that $A^{(3)}=\{0\}$. Let $a$ be an element of $A$. Define $a_1=a$, $a_{j+1}=a*a_j$, $j\ge 1$. Then, for every positive integer $n$,
  \[a^n=\sum_{k=1}^n\binom{n}{k}a_k=\binom{n}{1}a_1+\binom{n}{2}a_2+\binom{n}{3}a_3+\dots+\binom{n}{n-1}a_{n-1}+\binom{n}{n}a_n.\]
\end{proposition}
\begin{proof}
  We prove this proposition by induction on~$n$. For $n=1$, the result is clearly true. Suppose that
  \[a_1^n=\sum_{k=1}^n\binom{n}{k}a_k.\]
  Then, by using Lemma~\ref{lemma-4.2}, we have:
  \begin{align*}
    a^{n+1}&=a+a*a^n+a^n\\
           &=a+a*\left(\sum_{k=1}^n\binom{n}{k}a_k\right)+\sum_{k=1}^n\binom{n}{k}a_k\\
           &=a+\sum_{k=1}^n\binom{n}{k}(a*a_k)+\sum_{k=1}^n\binom{n}{k}a_k\\
           &=a+\sum_{k=1}^n\binom{n}{k}a_{k+1}+\sum_{k=1}^n\binom{n}{k}a_k\\
           &=a+\sum_{k=2}^{n+1}\binom{n}{k-1}a_k+\sum_{k=1}^n\binom{n}{k}a_k\\
           &=\sum_{k=1}^{n+1}\binom{n}{k-1}a_k+\sum_{k=1}^n\binom{n}{k}a_k\\
           &=\sum_{k=1}^{n+1}\binom{n+1}{k}a_k.
  \end{align*}
  We conclude that the result is true for all positive integers~$n$.
\end{proof}

\begin{proposition}\label{prop-1.8}
  Let $A$ be a left brace such that $A^{(3)}=\{0\}$. Let $a$ be an element of $A$. Write $a_1=a$, $a_{j+1}=a_1*a_j$ for $j\ge 1$. Then, for $j\ge 1$, we have that
\[(na)*a_j=\sum_{k=1}^n\binom{n}{k}a_{k+j}=\binom{n}{1}a_{j+1}+\binom{n}{2}a_{j+2}+\dots+\binom{n}{n-1}a_{j+n-1}+a_{j+n}.\]
\end{proposition}
\begin{proof}
  By Proposition~\ref{prop-1.7}, we have that
  $na=a^n+c$, where
  \[c=-\left(\binom{n}{2}a_2+\binom{n}{3}a_3+\dots+\binom{n}{n-1}a_{n-1}+a_n\right)\in A^2.\]
    By Corollary~\ref{corol-1.6}, we have that $(na)*x=a^n*x$ for every element $x\in A$. We argue by induction on~$n$. For $n=1$, the result is clear. Suppose that
    \[(a^n)*a_j=\sum_{k=1}^n\binom{n}{k}a_{k+j}.\]
    Then, by Lemma~\ref{lemma-4.2},
    \begin{align*}
      (a^{n+1})*a_j&=(aa^n)*a_j\\
                 &=a*(a^n*a_j)+a^n*a_j+a*a_j\\
                 &=a*\left(\sum_{k=1}^n\binom{n}{k}a_{k+j}\right)+\sum_{k=1}^n\binom{n}{k}a_{k+j}+a_{j+1}\\
                 &=a_{j+1}+\sum_{k=1}^n\binom{n}{k}(a*a_{k+j})+\sum_{k=1}^n\binom{n}{k}a_{k+j}\\
                 &=a_{j+1} +\sum_{k=1}^n\binom{n}{k}a_{k+j+1}+\sum_{k=1}^n\binom{n}{k}a_{k+j}\\
                 &=a_{j+1} +\sum_{k=2}^{n+1}\binom{n}{k-1}a_{k+j}+\sum_{k=1}^n\binom{n}{k}a_{k+j}\\
                 &=\sum_{k=1}^{n+1}\binom{n+1}{k}a_{k+j}.
    \end{align*}
    The result follows by induction.
  \end{proof}

\section{Proof of Theorem~\ref{th-mod-brace}}\label{sec-mod-braces}

In this section we prove Theorem~\ref{th-mod-brace}.

\begin{proof}[Proof of Theorem~\ref{th-mod-brace}]
  We show first that $(B, {\cdot})$ is a group. Clearly, $\cdot$ is an internal operation on $B$. In addition, given $x$, $y$, $z\in B$, we have that
\begin{align*}
  (xy)z&=xy+\lambda_{xy}(z)=(x+\lambda_x(y))+\lambda_{x+\lambda_x(y)}(z)= x+\lambda_x(y) +\lambda_x(\lambda_{\lambda_x(y)}(z))\\
  &=x+\lambda_x(y)+\lambda_x(\lambda_y(z))=x+\lambda_x(y+\lambda_y(z))=x+\lambda_x(yz)=x(yz).
\end{align*}
It follows that $\cdot$ is associative. Furthermore, $0x=0+\lambda_0(x)=0+x=x$ and $x0=x+\lambda_x(0)=x+0=x$ for each $x\in B$. Consequently, $0$ is the neutral element for $\cdot$. Finally, if $x\in B$, we have that $x(-\lambda_x^{-1}(x))=x+\lambda_x(-\lambda_x^{-1}(x))=x-\lambda_x(\lambda^{-1}_x(x))=x-x=0$, while $(-\lambda_x^{-1}(x))x=-\lambda_x^{-1}(x)+\lambda_{-\lambda_x^{-1}(x)}(x)=-\lambda_x^{-1}(x)+\lambda_{\lambda_x^{-1}(x)}^{-1}(x)=-\lambda_x^{-1}(x)+\lambda_x^{-1}(x)=0$, and so $x^{-1}=-\lambda_x^{-1}(x)$ is the inverse of $x$ under $\cdot$. Therefore $(B, {\cdot})$ is a group. Furthermore, if $x$, $y$, $z\in B$, we have that $xy-x+xz=x+\lambda_x(y)-x+x+\lambda_x(z)=x+\lambda_x(y)+\lambda_x(z)=x+\lambda_x(y+z)=x(y+z)$. Therefore, $(B, {+}, {\cdot})$ is a left brace.

Let $a$, $b\in B$. We have that $a*b=\lambda_a(b)-b$ and so $\lambda_{a*b}=\lambda_{\lambda_a(b)-b}=\lambda_{\lambda_a(b)}\circ\lambda_b^{-1}=\lambda_b\circ\lambda_b^{-1}=\operatorname{id_B}$. It follows that $a*b\in\Ker\lambda$. Therefore $B^{(2)}\subseteq\Ker\lambda$, as $\Ker\lambda=\Soc(B)$ and $B^{(2)}$ are ideals of $B$. Let $u\in B^{(2)}$ and $c\in B$, then $u*c=\lambda_u(c)-c=c-c=0$, which implies that $B^{(3)}=\{0\}$.
\end{proof}
\section{The subbrace generated by an element of a right nilpotent left brace with nilpotency class at most two}\label{sec-A(3)}
% For $n\in\N$, $n\ge 2$, we define $s_2=a*a$ and $s_{n+1}=a*a_n$, and $u_2=(-a)*(-a)$ and $u_{n+1}=(-a)*b_n$. We prove now that for $n\in\N$, $a^n=na+\sum_{i=2}^n\binom{n}{i}s_i$. For $n=1$, the result is trivial. For $n=2$, we have that $a^2=a+a*a+a=2a+a_2$. Suppose now that $a^n=na+\sum_{i=2}^n\binom{n}{i}s_i$. Then
% \begin{align*}
%   a^{n+1}&=aa^n=a+a*a^n+a^n\\
%          &=a+a*\left(na+\sum_{i=2}^n\binom{n}{i}a_i\right)+na+\sum_{i=2}^n\binom{n}{i}s_i\\
%          &=a+\sum_{i=1}^n\binom{n}{i}s_{i+1}+na+\sum_{i=2}^n\binom{n}{i}s_i\\
%          &=(n+1)a+\sum_{i=2}^{n+1}\binom{n}{i-1}s_i+\sum_{i=2}^n\binom{n}{i}s_i\\
%   &=(n+1)a+\sum_{i=2}^{n+1}\binom{n+1}{i}s_i.
% \end{align*}

% Now $0=a^{-1}a=a^{-1}+a^{-1}*a+a$, consequently, $a^{-1}=-a-(a^{-1}*a)=-a+(a^{-1}*(-a))$.  Let $u_2=a^{-1}*a^{-1}$, $u_{n+1}=a^{-1}*u_n$ for $n\ge 2$. The same argument used before with $a^{-1}$ instead of $a$ shows that $a^{-n}=na^{-1}+\sum_{i=2}^n\binom{n}{i}u_i=-na-n(a^{-1}*a)+\sum_{i=2}^n\binom{n}{i}u_i$. In particular, $a^{-n}\equiv na\pmod{A^2}$ for all $n\in\Z$.

% Given a left brace $(A, {+}, {\cdot})$, the map $\lambda\colon (A, {\cdot})\longrightarrow \Aut(A, {+})$ given by $\lambda(x)=\lambda_x$ for $x\in A$, where $\lambda_x(y)=-x+xy$, $x$, $y\in A$, is a group homomorphism.

Theorem~\ref{th-subbrace-1gen} describes the subbrace generated by an element of a right nilpotent left brace of right nilpotency class at most~$2$. We present in this section an alternative proof of this result as a rather easy consequence of Proposition~\ref{prop-3.13}.

\begin{proof}[Proof of Theorem~\ref{th-subbrace-1gen}]
By Proposition~\ref{prop-3.13}, the map $\lambda\colon (A, {+})\longrightarrow \Aut(A, {+})$ is a group homomorphism and $A^2\subseteq \Ker(\lambda)$. In particular, for $i\in\Z$,
\[\lambda_{a_i}=\lambda_{-a^i+a^{i+1}}=\lambda_{a^i}^{-1}\lambda_{a^{i+1}}=\lambda_{(a^i)^{-1}a^{i+1}}=\lambda_a.\]
Clearly, $B$ is contained in the left subbrace generated by~$a$. Furthermore, $B$ is closed under the addition and by taking additive opposites.
If $x=\sum_{i\in\Z}x_ia_i\in B$, then
\[\lambda_{\sum_{i\in\Z}x_ia_i}=\prod_{i\in\Z}\lambda_{a_i}^{x_i}=\prod_{i\in\Z}\lambda_a^{x_i}=\lambda_{a}^t, \qquad\text{where $t={\sum_{i\in\Z}x_i}$.}\]
Here the product symbol corresponds to the iterated composition of group homomorphisms.
It follows that, for $y=\sum_{i\in\Z}y_ia_i\in B$,
\[\lambda_{\sum_{i\in\Z}x_ia_i}\Bigl(\sum_{j\in\Z}y_ja_j\Bigr)=\sum_{j\in\Z}y_ja_{j+t}=\sum_{j\in\Z}y_{j-t}a_j,\qquad t=\sum_{i\in\Z}x_i.\]
Therefore, $xy=x+\lambda_x(y)\in B$.

Furthermore, if $x=\sum_{i\in \Z}x_ia_i$, then $y=\sum_{i\in\Z}(-x_{i-\sum_{j\in\Z}x_j})a_i\in B$ satisfies that $xy=yx=0$. Consequently $y=x^{-1}$ is the multiplicative inverse of~$x$. We conclude that $B$ is a subbrace, and so $B$ is the subbrace generated by~$a$.

We note that if $x=\sum_{i\in \Z}x_ia_i\in B$ and $y=\sum_{i\in I}y_ia_i$, then $x*y=\lambda_x(y)-y=\sum_{i\in \Z}(y_{i-\sum_{j\in \Z}x_i}-y_i)a_i$ satisfies that $\sum_{i\in\Z}(y_{i-\sum_{j\in \Z}x_i}-y_i)=\sum_{i\in\Z}y_i-\sum_{i\in\Z}y_i=0$. It follows that $B*B\subseteq D$. Furthermore, if $x=\sum_{i\in\Z}x_ia_i\in D$, then $x=\sum_{i\in\Z}x_ia_i-\sum_{i\in \Z}(-x_ia)=\sum_{i\in\Z}x_i(a_i-a)= \sum_{i\in\Z}x_i(\lambda_{a^i}(a)-a)=\sum_{i\in\Z}x_i(a^i*a)\in C*C$. It follows that $D=C*C$. %Moreover, if $x=\sum_{i\in \Z}x_ia_i\in D$ and $y=\sum_{j\in\Z}y_ja_j\in C$, then $x*y=\lambda_x(y)-y=y-y=0$ as $\sum_{i\in\Z}x_i=0$. Consequently, $D*C=0$, that is, $C^{(3)}=\{0\}$.
\end{proof}

\section{A free one-generated right nilpotent left brace of class two}\label{sec-A(3)-free}

We prove Theorem~\ref{th-free-A(3)} as an application of Theorem~\ref{th-mod-brace}.

\begin{proof}[Proof of Theorem~\ref{th-free-A(3)}]
  As $\sum_{i\in\Z}y_{i-t}=\sum_{i\in \Z}y_i$ for each $t\in\Z$, we obtain that $\lambda_{\lambda_x(y)}=\lambda_y$ for all $x$, $y\in C$ and so we conclude that $(C, {+}, {\cdot})$ is a left brace by Theorem~\ref{th-mod-brace}. In particular, for each $a$, $b\in C$, $-a+ab=-a+a+\lambda_a(b)$ and the map $\lambda\colon (C, {\cdot})\longrightarrow \Aut(C, {+})$ is also a group homomorphism.

  If $i=0$, $a_0=\lambda_{a^0}(a)=\lambda_{0}(a)=a=c_0$. Assume that $i>0$. Then $a_i=\lambda_{a^i}(a)=\lambda_a\circ\overset{(i)}{\cdots}\circ \lambda_a(a)=\lambda_{ia}(c_0)=c_i$ and $a_{-i}=\lambda_{a^{-i}}(a)=\lambda_a^{-1}\circ\overset{(i)}{\cdots}\circ \lambda_a^{-1}=\lambda_{(-i)a}(c_0)=c_{-i}$. Consequently, for every $i\in\Z$, $\lambda_{a^i}(a)=c_i$ and so $C$ coincides with the subbrace generated by $a=c_0$ by Theorem~\ref{th-subbrace-1gen}.

  Let $s_1=a$, $s_j=a*s_{j-1}$, $j\ge 2$. Let us show by induction that $s_j=\sum_{k=0}^{j-1}(-1)^k\binom{j-1}{k}c_{j-k-1}$ for $j\ge 2$. For $j=2$, we have that $s_2=a*a=\lambda_a(a)-a=c_1-c_0=-c_0+c_1$. Suppose that $s_j=\sum_{k=0}^{j-1}(-1)^k\binom{j-1}{k}c_{j-k-1}$. Then
  \begin{align*}
    s_{j+1}&=a*s_j-s_j=a*\sum_{k=0}^{j-1}(-1)^k\binom{j-1}{k}c_{j-k-1}-\sum_{k=0}^{j-1}(-1)^k\binom{j-1}{k}c_{j-k-1}\\
           &=\sum_{k=0}^{j-1}(-1)^k\binom{j-1}{k}(a*c_{j-k-1})-\sum_{k=0}^{j-1}(-1)^k\binom{j-1}{k}c_{j-k-1}\\
           &=\sum_{k=0}^{j-1}(-1)^k\binom{j-1}{k}c_{j+1-k-1}-\sum_{k=0}^{j-1}(-1)^k\binom{j-1}{k}c_{j-k-1}\\
           &=\sum_{k=0}^{j-1}(-1)^k\binom{j-1}{k}c_{j+1-k-1}-\sum_{k=1}^{j}(-1)^k\binom{j-1}{k-1}c_{j+1-k-1}\\
    &=\sum_{k=0}^{j}(-1)^k\binom{j}{k}c_{j+1-k-1}.
  \end{align*}
  We conclude that $0\ne s_j\in C^j$ for all $j\in\N$ and so $C^j\ne \{0\}$ for all $j\in\N$.

  The description of the subbrace generated by one element given in Theorem~\ref{th-subbrace-1gen} shows that, with the notation of that theorem, the assignment
  \[\sum_{i\in \Z}x_ic_i\longmapsto \sum_{i\in \Z}x_ia_i\]
  is a left brace epimorphism between $C$ and the subbrace of $A$ generated by $a$ and maps $c_0$ to $a$. Furthermore it is clear that it is the unique left brace epimorphism satisfying this condition.
\end{proof}

\section{The subbrace generated by one element of a Smoktunowicz-nilpotent left brace}\label{sec-Smoktunowicz}

In this section, we will consider a left brace $(A, {+}, {\cdot})$ satisfying  $A^{m+1}=A^{(3)}=\{0\}$ for a natural number~$m\ge 2$. These left braces have been studied for the first time by  Smoktunowicz in \cite{Smoktunowicz18-tams}. Our aim is to give a description of the subbrace $S$ generated by one element $(A, {+}, {\cdot})$ that gives information about the left nilpotent series $S^k$ of $S$.

Given $a\in A$, if we define $a_1=a$, $a_{j+1}=a*a_j$ for $j\ge 1$, we see that $a_j=0$ for $j\ge m+1$. Since $\binom{n}{k}=0$ when $0<n<k$, by Proposition~\ref{prop-1.8} we have:
\begin{equation}
  (na)*a_j=\sum_{k=1}^{m-j}\binom{n}{k}a_{k+j}.
  \label{eq-na-pos}
\end{equation}
Our next aim is to prove that this expression is also valid for all integers~$n$, even if $n<0$. Let us start with $n=-1$.
\begin{proposition}\label{prop-minus-1}
  Let $A$ be a left brace such that $A^{(3)}=A^{m+1}=\{0\}$ for a natural number $m\ge 2$. Define $a_1=a$, $a_{j+1}=a*a_j$ for $1\le j\le m-1$. If $1\le k\le m-1$, then
  \[(-a)*a_{m-k}=\sum_{j=1}^k (-1)^j a_{m-k+j}.\]
\end{proposition}
\begin{proof}
  We argue by induction on $k$. Note that, for all $k$, we have that
  \begin{equation}
    0=((-a)+a)*a_{m-k}=(-a)*a_{m-k+1}+(-a)*a_{m-k}+a_{m-k+1}
    \label{eq-minus-a}
  \end{equation}
  by Proposition~\ref{prop-1.5}.
  
  We start with $k=1$. In this case, we have that $(-a)*a_{m}\in A^{m+1}=\{0\}$ and so $(-a)*a_{m-1}=-a_m$, as desired.

  Suppose now that the result is valid for $k-1$, that is,
  \[(-a)*a_{m-k+1}=\sum_{j=1}^{k-1}(-1)^ja_{m-k+1+j}.\]
  By Equation~\eqref{eq-minus-a}, we have that
  \begin{align*}
    (-a)*a_{m-k}&=-a_{m-k+1}-(-a)*a_{m-k+1}\\
                &=-a_{m-k+1}+\sum_{j=1}^{k-1}(-1)^{j+1}a_{m-k+1+j}\\
                &=-a_{m-k+1}+\sum_{j=2}^k(-1)^ja_{m-k+j}\\
                &=\sum_{j=1}^k(-1)^ja_{m-k+j}.
  \end{align*}
  The result follows by induction.
\end{proof}

Now we can prove that Equation~\eqref{eq-na-pos} is valid for all integers~$n$.
\begin{proposition}\label{prop-star-int}
  Let $A$ be a left brace such that $A^{(3)}=A^{m+1}=\{0\}$. Define $a_1=a$, $a_{j+1}=a*a_j$ for $1\le j\le m-1$. Let $n$ be an integer. Then
\[(na)*a_j=\sum_{k=1}^{m-j}\binom{n}{k}a_{k+j}.\]
\end{proposition}
\begin{proof}
  The result is obvious for $n=0$ and known for positive integers $n$ by Equation~\eqref{eq-na-pos}. Hence it is enough to prove it for negative integers. We argue by induction on $-n$. When $n=-1$, we observe that $\binom{-1}{j}=(-1)^j\binom{j}{j}=(-1)^j$. The result follows by Proposition~\ref{prop-minus-1}.

  Assume that the result is true for $-n=k$, that is,
\[(na)*a_j=\sum_{k=1}^{m-j}\binom{n}{k}a_{k+j}.\]
  We want to prove that it is also true for $-n=k+1$. By Proposition~\ref{prop-1.5}, Lemma~\ref{lemma-4.2}, Corollary~\ref{corol-1.6}, and Proposition~\ref{prop-star-int}, we have that
  \begin{align*}
    ((-1&+n)a)*a_j=(-a)*((na)*a_j)+(-a)*a_j+(na)*a_j\\
                 &=(-a)*\left(\sum_{k=1}^{m-j}\binom{n}{k}a_{k+j}\right)+\sum_{k=1}^{m-j}\binom{-1}{k}a_{k+j}+\sum_{k=1}^{m-j}\binom{n}{k}a_{k+j}\\
                 &=\sum_{k=1}^{m-j}\binom{n}{k}((-a)*a_{k+j})+\sum_{k=1}^{m-j}\left(\binom{-1}{k}+\binom{n}{k}\right)a_{j+k}\\
                 &=\sum_{k=1}^{m-j}\binom{n}{k}\sum_{t=1}^{m-k-j}\binom{-1}{t}a_{t+k+j} + \sum_{k=1}^{m-j}\left(\binom{-1}{k}+\binom{n}{k}\right)a_{k+j}\\
                 &=\sum_{r=2}^{m-j}\biggl(\sum_{\substack{1\le k\le m-j\\1\le t\le m-k-j\\t+k=r}}\binom{n}{k}\binom{-1}{t}\biggr)a_{r+j}+\sum_{k=1}^{m-j}\left(\binom{-1}{k}+\binom{n}{k}\right)a_{k+j}\\
                 &=\sum_{r=2}^{m-j}\biggl(\sum_{t=1}^{r-1}\binom{n}{r-t}\binom{-1}{t}\biggr)a_{r+j}+\sum_{k=1}^{m-j}\left(\binom{-1}{k}+\binom{n}{k}\right)a_{k+j}\\
    &=\sum_{r=2}^{m-j}\biggl(\sum_{t=0}^r\binom{n}{r-t}\binom{-1}{t}\biggr)a_{r+j}+\left(\binom{-1}{1}+\binom{n}{1}\right)a_{1+j}\\
    &=\sum_{r=2}^{m-j}\binom{n-1}{r}a_{r+j}+(-1+n)a_{1+j}\\
    &=\sum_{r=1}^{m-j}\binom{n-1}{r}a_{r+j}=\sum_{k=1}^{m-j}\binom{n-1}{k}a_{k-j}.
  \end{align*}
By induction, the result is valid for all negative integers~$n$.
\end{proof}

We are now in a position to give a description of the subbrace generated by an element $a$ in a left brace with $A^{(3)}=A^{m+1}=\{0\}$.
\begin{proof}[Proof of Theorem~\ref{th-brace-1gen}]
  It is clear that if $T$ is a subbrace of $A$ containing $a$, then $S\subseteq T$. Observe also that $a_k=0$ if $k>m$. It is enough to check that $S$ is a subbrace of $A$, that is, $S$ is closed by taking sums, products, additive opposites, and multiplicative inverses. It is clear that $S$ is closed by taking sums and additive opposites. 

Since $xy=x+x*y+y$ for every $x$, $y\in A$ and $S$ is closed by taking sums, it will be closed under taking products if it is closed under taking the star operation $*$. Consider
\[t=\sum_{k=1}^m t_ka_k\in S,\quad u=\sum_{k=1}^mu_ka_k\in S.\]
Since $\sum_{k=2}^mt_ka_k\in A^2$, by Corollary~\ref{corol-1.6} and Lemma~\ref{lemma-4.2} we see that
\[t*u=(t_1a_1)*u=\sum_{k=1}^mu_k (t_1a)*a_k.\]
By Proposition~\ref{prop-star-int}, this is a linear combination with integer coefficients of elements of~$S$, and so $t*u\in S$.

To conclude the proof, we must show that $S$ is closed under taking inverses. Take $t=\sum_{i=1}^mt_ia_i$, we will show that
\[t^{-1}=\sum_{j=1}^m\left(-\sum_{k=0}^{j-1}\binom{-t_1}{k}t_{j-k}\right)a_j.\]
Call $n$ the right hand side of the previous equality.
Since $nt=n+n*t+t$, we compute
\begin{align*}
  n*t&=\left(-\binom{-t_1}{0}t_1a_1\right)*\sum_{i=1}^mt_ia_i\\
     &=(-t_1a_1)*\sum_{i=1}^mt_ia_i\\
     &=\sum_{i=1}^mt_i((-t_1a_1)*a_i)\\
     &=\sum_{i=1}^mt_i\sum_{k=1}^{m-i}\binom{-t_1}{k}a_{i+k}\\
     &=\sum_{r=1}^m\biggl(\sum_{\substack{1\le i\le m\\1\le k\le m-i\\i+k=r}}t_i\binom{-t_1}{k}\biggr)a_r\\
     &=\sum_{r=1}^m\biggl(\sum_{k=1}^{r-1}t_{r-k}\binom{-t_1}{k}\biggr)a_r.  
\end{align*}
by Lemma~\ref{lemma-4.2}, Corollary~\ref{corol-1.6} and Proposition~\ref{prop-star-int}. It follows that $nt=n+n*t+t=0$, the neutral element of $\cdot$, and so $n$ is the multiplicative inverse of $t$ in $A$. As $n\in S$ and the inverse is unique, $S$ is closed under taking inverses. This completes the proof.
\end{proof}

\begin{remark}\label{rem-left}
  We note that in Theorem~\ref{th-brace-1gen}, if for $1\le r\le m$ we call
  \[S_r=\Bigl\{\sum_{k=r}^mt_ka_k\mid t_k\in \Z, r\le k\le m\Bigr\},\]
  we have that $S_r\subseteq S^r$ as $a_k\in S^r$ for $r\le k\le m$.

  We prove now by induction on $r$ that $S^r\subseteq S_r$. For $r=1$, the result is clear since $S=S_1$. Suppose that $S^r\subseteq S_r$. Then $S^{r+1}=S*S^r\subseteq S*S_r$ and, since
  \begin{align*}
    \left(\sum_{k=1}^mv_ka_k\right)*\left(\sum_{k=r}^mt_ka_k\right)&=v_1a*\left(\sum_{k=r}^mt_ka_k\right)\\&=\sum_{k=r}^mv_1t_ka_{k+1}=\sum_{k=r+1}^mv_1t_{k-1}a_{k}\in S_{r+1},
  \end{align*}
  we conclude that $S^{r+1}\subseteq S_{r+1}$. By induction, we have that $S_r=S^r$ for $1\le r\le m$. For $r\ge m+1$, we have that $S^r=\{0\}$.
\end{remark}

\section{Free one-generated Smoktunowicz-nilpotent left braces with right nilpotency class two}\label{sec-Smoktunowicz-free}
Our next step is to prove Theorem~\ref{th-Smoktunowicz-free}, that gives the construction of a free one-generated left brace $B_m$ with $B_m^{(3)}=B_m^{m+1}=\{0\}$ for a natural number $m\ge 2$. As we have mentioned in the introduction, we will do it by means of the construction in Theorem~\ref{th-mod-brace}.

\begin{proof}[Proof of Theorem~\ref{th-Smoktunowicz-free}]
  Consider an integer $m\ge 2$. We define in
  \[B_m=\overbrace{\mathbb{Z}\times\dots\times\mathbb{Z}}^{(m)}\]
  the addition in the usual form
  \[(n_1,\dots, n_m)+(t_1,\dots, t_m)=(n_1+t_1,\dots,n_m+t_m).\]
  Given $\mathbf{n}=(n_1,\dots, n_m)\in B_m$ and $\mathbf{t}=(t_1,\dots, t_m)\in B_m$, we define $\lambda_{\mathbf{n}}(\mathbf{t})=(z_1,\dots, z_m)$ where, for $1\le i\le m$,
\[z_i=\sum_{k=0}^{i-1}\binom{n_1}{k}t_{i-k}.\]
Let us denote $z_i$ by $(\lambda_{\mathbf{n}}(\mathbf{t}))_i$, $1\le i\le m$.
It is clear that $\lambda_{\mathbf{n}}$ is an endomorphism of $(B, {+})$ for all $\mathbf{n}\in B_m$. We prove that $\lambda$ is a homomorphism from $(B, {+})$ to $\operatorname{End}(B, {+})$. Consider the elements of $B_m$ $\mathbf{n}=(n_1,\dots, n_m)$, $\mathbf{t}=(t_1,\dots, t_m)$, and $\mathbf{w}=(w_1,\dots, w_m)$. Then
\[(\lambda_{\mathbf{w}}(\mathbf{t}))_1=\sum_{k=0}^0\binom{w_1}{0}t_{1-k}=t_1\]
and so, for $1\le i\le m$,
\begin{align*}
  (\lambda_{\mathbf{n}}(\lambda_{\mathbf{w}}(\mathbf{t})))_i&=  \sum_{k=0}^{i-1}\binom{(\lambda_{\mathbf{w}}(\mathbf{t}))_1}{k}(\lambda_{\mathbf{w}}(\mathbf{t}))_{i-k}\\
                             &=\sum_{k=0}^{i-1}\binom{t_1}{k}\sum_{l=0}^{i-k-1}\binom{w_1}{l}t_{i-k-l}\\
                             &=\sum_{k=0}^{i-1}\sum_{l=0}^{i-k-1}\binom{t_1}{k}\binom{w_1}{l}t_{i-k-l}\\
                             &=\sum_{k=0}^{i-1}\sum_{r=k}^{i-1}\binom{t_1}{k}\binom{w_1}{r-k}t_{i-r}\\
  &=\sum_{k=0}^{r-1}\binom{t_1+w_1}{r}t_{i-r}=(\lambda_{\mathbf{n}+\mathbf{w}}(\mathbf{t}))_i.
\end{align*}
We also note that if $\mathbf{0}=(0,\dots, 0)\in B_m$, then
\[(\lambda_{\mathbf{0}}(\mathbf{t}))_i=\sum_{k=0}^{i-1}\binom{0}{k}t_{i-k}=t_i,\]
therefore $\lambda_{\mathbf{0}}=\id$. Consequently, if $\mathbf{n}\in B_m$, we have that $\lambda_{\mathbf{n}}$ has an inverse $\lambda_{-\mathbf{n}}$ and so the image of $\lambda$ is contained in the automorphism group of $(B, {+})$. Furthermore, for $\mathbf{n}$, $\mathbf{t}$, $\mathbf{w}\in B_m$ and $1\le i\le m$,
\[(\lambda_{\lambda_{\mathbf{w}}(\mathbf{n})}(\mathbf{t}))_i=\sum_{k=0}^{i-1}\binom{(\lambda_{\mathbf{w}}(\mathbf{n}))_1}{k}t_{i-k}=\sum_{k=0}^{i-1}\binom{n_1}{k}t_{i-k}=(\lambda_{\mathbf{n}}(\mathbf{t}))_i.\]
We conclude that $\lambda_{\lambda_{\mathbf{w}}(\mathbf{n})}=\lambda_{\mathbf{n}}$.

By Theorem~\ref{th-mod-brace}, we have that $B_m$ is a left brace with the product $\cdot$ given by $\mathbf{n}\mathbf{t}=\mathbf{n}+\lambda_{\mathbf{n}}(\mathbf{t})$ for $\mathbf{n}$, $\mathbf{t}\in B_m$ and $B^{(3)}=\{0\}$.

Let $\mathbf{b}=(1,0,\dots, 0)$. If $\mathbf{b}_1=b$ and $\mathbf{b}_{k+1}=\mathbf{b}*\mathbf{b}_k$ for $k\ge 1$, we prove that ${(\mathbf{b}_i)}_j=\delta_{i,j}$ for $1\le j\le m$ by induction on~$i$. The result is clear for $i=1$. If ${(\mathbf{b}_i)}_j=\delta_{i,j}$ for $1\le j\le m$ and some $i$, then
\begin{align*}
  {(\mathbf{b}_{i+1})}_j&={(\mathbf{b}*\mathbf{b}_{i})}_j={(\lambda_{\mathbf{b}}(\mathbf{b}_i)-\mathbf{b}_i)}_j\\
                        &=\sum_{k=0}^{j-1}\binom{1}{k}{(\mathbf{b}_i)}_{j-k}-\delta_{i,j}=\binom{1}{0}\delta_{i,j}+\binom{1}{1}\delta_{i,j-1}-\delta_{i,j}\\
                        &=\delta_{i,j-1}=\delta_{i+1,j}.
\end{align*}
By Theorem~\ref{th-brace-1gen}, we have that the subbrace of $B_m$ generated by $\mathbf{b}$ coincides with $B_m$. By Remark~\ref{rem-left}, we obtain that
\[B_m^r=\{(t_1,\dots, t_m)\in B_m\mid \text{$t_i=0$ for $1\le i<m$}\}\]
if $r\le m$, and $B_m^r=\{\mathbf{0}\}$ if $r>m$.

It is clear that, given a left brace $A$ with $A^{(3)}=A^{m+1}=\{0\}$ for a natural number $m\ge 2$, with the same notation as in Theorem~\ref{th-brace-1gen}, the map $(n_1,\dots, n_m)\longmapsto \sum_{k=1}^mn_ka_k$ defines a left brace epimorphism $\varphi$ from $B_m$ to the subbrace of $A$ generated by $a$ such that $\varphi(\mathbf{b})=a$, and this is the unique possible homomorphism with this condition.
\end{proof}
\section*{Acknowledgements}
The third author is very grateful to the Conselleria d'Innovaci\'o, Universitats, Ci\`encia i Societat Digital of the Generalitat (Valencian Community, Spain) and the Universitat de Val\`encia for their financial support and grant to host researchers affected by the war in Ukraine in research centres of the Valencian Community. He is sincerely grateful to the first and second authors for their hospitality, support and care. The third author is also grateful for the support of the Isaac Newton Institute for Mathematical Sciences and the University of Edinburgh provided in the frame of LMS Solidarity Supplementary Grant Programme. He is sincerely grateful to Agata Smoktunowicz.

We also thank  Lorenzo Stefanello for some interesting conversations that have helped us to improve the presentation of this paper.

\bibliographystyle{plain}
\bibliography{bibgroup}

\providecommand{\iflanguage}[3]{#3}\relax \providecommand{\mathfrak}{\mathcal}
  \providecommand{\acceptedin}{\iflanguage{spanish}{Aceptado en}{Accepted in}}
  \providecommand{\accepted}{\iflanguage{spanish}{Aceptado}{Accepted}}
  \providecommand{\talkat}{\iflanguage{spanish}{Charla en}{Talk at}}
  \providecommand{\plenarytalkat}{\iflanguage{spanish}{Charla plenaria
  en}{Plenary talk at}} \providecommand{\JIF}{JIF}
  \providecommand{\PhDcourseat}{\iflanguage{spanish}{Curso de doctorado en}{PhD
  course at}} \providecommand{\transrus}{\iflanguage{spanish}{Traducci{\'o}n
  del art{\'\i}culo original en ruso en}{Translation of the original paper in
  Russian from}}
  \providecommand{\supervisedby}{\iflanguage{spanish}{Direcci{\'o}n:}{Supervised
  by}} \providecommand{\visited}{\iflanguage{spanish}{visitada}{visited}}
  \providecommand{\inpress}{\iflanguage{spanish}{en
  prensa}{\iflanguage{catalan}{en premsa}{in press}}}
  \providecommand{\encurs}{\iflanguage{spanish}{en
  curso}{\iflanguage{catalan}{en curs}{in course}}}
  \providecommand{\aand}{\iflanguage{spanish}{y}{and}}
  \providecommand{\invitedtalkat}{\iflanguage{spanish}{Conferencia invitada
  en}{Invited talk at}} \providecommand{\posterat}{\iflanguage{spanish}{Póster
  en}{Poster at}}
  \providecommand{\invitedtalk}{\iflanguage{spanish}{Conferencia
  invitada}{Invited talk}}
  \providecommand{\oral}{\iflanguage{spanish}{Comunicaci{\'o}n oral}{Oral
  communication}}
  \providecommand{\Poland}{\iflanguage{spanish}{Polonia}{Poland}}
  \providecommand{\Germany}{\iflanguage{spanish}{Alemania}{Germany}}
  \providecommand{\Italy}{\iflanguage{spanish}{Italia}{Italy}}
  \providecommand{\Ireland}{\iflanguage{spanish}{Irlanda}{Ireland}}
  \providecommand{\Hungary}{\iflanguage{spanish}{Hungr{\'\i}a}{Hungary}}
  \providecommand{\Portugal}{\iflanguage{spanish}{Portugal}{Portugal}}
  \providecommand{\Cairo}{\iflanguage{spanish}{El Cairo, Egipto}{Cairo, Egypt}}
  \providecommand{\Bath}{\iflanguage{spanish}{Bath, Inglaterra, Reino
  Unido}{Bath, England, United Kingdom}}
  \providecommand{\presentada}[2]{\iflanguage{spanish}{presentada por
  #1}{presented by #2}}
  \providecommand{\Ponenciaplenaria}{\iflanguage{spanish}{Ponencia
  plenaria}{Plenary talk}}
  \providecommand{\Naples}{\iflanguage{spanish}{N{\'a}poles}{Naples}}
  \def\cprime{$'$} \providecommand{\germ}{\mathfrak}
\begin{thebibliography}{10}

\bibitem{BardakovNeshchadimYadav22-jpaa}
V.~G. Bardakov, M.~V. Neshchadim, and M.~K. Yadav.
\newblock On {${\lambda}$}-homomorphic skew braces.
\newblock {\em J.\ Pure Appl.\ Algebra}, 226:106961 (37 pages), 2022.

\bibitem{BonattoJedlicka-central-nilp-skew-braces}
M.~Bonatto and P.~Jedli{\v c}ka.
\newblock Central nilpotency of skew braces.
\newblock {\em J. Algebra Appl.}, 2022.
\newblock \url{https://doi.org/10.1142/S0219498823502559}.

\bibitem{Caranti20}
A.~Caranti.
\newblock Bi-skew braces and regular subgroups of the holomorph.
\newblock {\em J. Algebra}, 562:647--665, 2020.

\bibitem{CarantiStefanello21}
A.~Caranti and L.~Stefanello.
\newblock From endomorphisms to bi-skew braces, regular subgroups, the
  {Y}ang-{B}axter equation, and {H}opf-{G}alois structures.
\newblock {\em J. Algebra}, 587:462--487, 2021.

\bibitem{Cedo18}
F.~Ced{\'o}.
\newblock Left braces: solutions of the {Y}ang-{B}axter equation.
\newblock {\em Adv. Group Theory Appl.}, 5:33--90, 2018.

\bibitem{Childs19}
L.~N. Childs.
\newblock Bi-skew braces and {H}opf-{G}alois structures.
\newblock {\em New York J. Math.}, 25:574--588, 2019.

\bibitem{Gould72}
H.~W. Gould.
\newblock {\em Combinatorial identities}.
\newblock Morgantown Printing and Binding Co., Morgantown, WV, USA, 1972.

\bibitem{GuarnieriVendramin17}
L.~Guarnieri and L.~Vendramin.
\newblock Skew-braces and the {Y}ang-{B}axter equation.
\newblock {\em Math.\ Comp.}, 86(307):2519--2534, 2017.

\bibitem{JespersVanAntwerpenVendramin22}
E.~Jespers, A.~Van Antwerpen, and L.~Vendramin.
\newblock Nilpotency of skew braces and multipermutation solutions of the
  {Y}ang-{B}axter equation.
\newblock {\em Commun. Contemp. Math.}, 2022.
\newblock \url{https://doi.org/10.1142/S021919972250064X}.

\bibitem{Koch21}
Alan Koch.
\newblock Abelian maps, bi-skew braces, and opposite pairs of {H}opf-{G}alois
  structures.
\newblock {\em Proc. Amer. Math. Soc. Ser. B}, 8:189--203, 2021.

\bibitem{Rump07}
W.~Rump.
\newblock Braces, radical rings, and the quantum {Y}ang-{B}axter equation.
\newblock {\em J. Algebra}, 307:153--170, 2007.

\bibitem{Rump20}
W.~Rump.
\newblock One-generator braces and indecomposable set-theoretic solutions to
  the {Y}ang-{B}axter equation.
\newblock {\em Proc. Edinburgh Math. Soc.}, 63:676--696, 2020.

\bibitem{Smoktunowicz18-tams}
A.~Smoktunowicz.
\newblock On {E}ngel groups, nilpotent groups, rings, braces and the
  {Y}ang-{B}axter equation.
\newblock {\em Trans.\ Amer.\ Math.\ Soc.}, 370(9):6535--6564, 2018.

\bibitem{SmoktunowiczSmoktunowicz18}
A.~Smoktunowicz and A.~Smoktunowicz.
\newblock Set-theoretic solutions of the {Y}ang-{B}axter equation and new
  classes of {R}-matrices.
\newblock {\em Linear Algebra Appl.}, 546:86--114, 2018.

\bibitem{StefanelloTrappeniers23-jpaa}
L.~Stefanello and S.~Trappeniers.
\newblock On bi-skew braces and brace blocks.
\newblock {\em J.\ Pure Appl.\ Algebra}, 227:107295, 2023.

\end{thebibliography}
\end{document}